\documentclass[12pt, a4paper]{article}
\usepackage[english]{babel}
\usepackage[utf8x]{inputenc}
\usepackage{amsmath,amsfonts,stackrel,amsthm,enumerate}
\usepackage{graphicx,setspace,mathrsfs}
\usepackage[margin=1in]{geometry}
\newtheorem{definition}{Definition}
\newtheorem{theorem}{\bf Theorem}[section]
\newtheorem{remark}{\bf Remark}[section]
\newtheorem{proposition}{Proposition}[section]
\newtheorem{lemma}{Lemma}[section]
\newtheorem{corollary}{Corollary}[section]

\usepackage{chngcntr}

\usepackage[colorinlistoftodos]{todonotes}
\date{}
\title{Subordinated Compound Poisson processes of order $k$}
\author{Ayushi Singh Sengar and N. S. Upadhye*\\{\small Department of Mathematics, Indian Institute of Technology Madras}\\
\small{Chennai-600036, Tamil Nadu, INDIA.}\\
\small{*e-mail:neelesh@iitm.ac.in}}
\begin{document}
	\maketitle

\begin{abstract}
\noindent
In this article, the compound Poisson processes of order $k$ (CPPoK) is introduced and its properties are discussed. Further, using mixture of tempered stable subordinator (MTSS) and its right continuous inverse, the two subordinated CPPoK with various distributional properties are studied. It is also shown that the space and tempered space fractional versions of CPPoK and PPoK can be obtained, which generalize the results in the literature for the process defined in \cite{sfpp}.
\end{abstract}

\section{Introduction}
\noindent
The Poisson distribution has been the conventional model for count data analysis, and due to its popularity and applicability various researchers have generalized it in several directions; e.g. compound Poisson processes, fractional (time-changed) versions of Poisson processes (see \cite{fell,lask,beghinejp2009}, and references therein). A handful of researchers have also studied the distributions and processes of order $k$ (see \cite{distribution,Filippu1984}). In particular, the discrete distribution of order $k$, introduced by Philippou et al. (see \cite{phili83-geo}), include binomial, geometric and negative binomial distributions of order $k$. These distributions play an important role in several areas, such as, reliability, statistics, finance, actuarial risk analysis (see \cite{Mandelbrot-Fisher-Cal97,TCPPoK,BOWERS}).

\vskip 2ex
\noindent
In risk theory, the total claim amount is usually modelled by using a compound Poisson process, say $Z_t = \sum_{i=1}^{N(t)} Y_i$, where the compounding random variables $Y_i$ are iid and the number of claims $N(t)$, independent of $\{Y_i\}_{i\ge 1}$, follow Poisson distribution. But, due to the restriction of single arrival in each inter-arrival time, the model is not suitable to use. Kostadinova and Minkova \cite{Poiss-order-k} introduced a Poisson process of order $k$ (PPoK), which allows us to model arrival in a group of size $k$. Recently, a time-changed version of Poisson processes of order $k$ is studied by \cite{TCPPoK} which allows group arrivals and also the case of extreme events, which is not covered by \cite{Poiss-order-k}. In spite of its applicability, this model is still not suitable for underdispersed dataset. Therefore, a generalization of this model is essential and is proposed in this article.

\vskip 2ex
\noindent
To the best of our knowledge, such a generalization is not yet studied. Therefore, we introduce the compound Poisson process of order $k$ (CPPoK) with the help of the Poisson process of order $k$ (PPoK) and study its distributional properties. Then, we time-change CPPoK with a special type of L\'evy subordinator known as mixture of tempered stable subordinator, and its right continuous inverse, and analyze some properties of these time-changed processes.  

\vskip 2ex
\noindent
The article is organized as follows. In section \ref{sec:cppok}, we introduce CPPoK and derive some of its general properties along with  martingale characterization property. In section \ref{sec:subordinator}, we introduce two types of CPPoK with the help of MTSS and its right continuous inverse, and derive some important distributional properties.


\section{Compound Poisson Process of order $k$ and its properties}\label{sec:cppok}
\noindent
In this section, we introduce CPPoK and derive its distributional properties. First, we define the Poisson distribution of order $k$.
\begin{definition}[Philippou \cite{Filippu1984}]\label{Def:pok}
Let $N^{(k)}\sim PoK(\lambda)$, the Poisson distribution of order $k$ (PoK) with rate parameter $\lambda >0$, then the probability mass function	 (pmf) of $N^{(k)}$ is given by $$ \mathbb{P}[N^{(k)} = n] =  \sum_{\substack{x_1,x_2,...,x_k\geq 0\\\sum_{j=1}^{k}jx_j=n}}e^{-k\lambda} \frac{\lambda^{x_1+x_2+..+x_k}}{x_1!x_2!...x_k!},~ n=0,1,\ldots,$$
where the summation is taken over all non-negative integers $x_1,x_2,\ldots,x_k$ such that $x_1+2x_2+\ldots +kx_k = n$.
\end{definition}

\noindent
Philippou \cite{Filippu1984} showed the existence of PoK as a limiting distribution of negative binomial distribution of order $k$. Kostadinova and Minkova \cite{Poiss-order-k} later generalized PoK to evolve over time, in terms of, a process which can be defined as follows.
	
\begin{definition}[Kostadinova and Minkova \cite{Poiss-order-k}]\label{ppok}
Let $\{N(t)\}_{t\geq 0}$ denote $PP(k \lambda)$, the Poisson process with rate parameter $k\lambda$, and $\{X_i\}_{i\geq 1}$ be a sequence of independent and identically distributed (IID) discrete uniform random variables with support on $\{1,2,\ldots,k\}$. Also, assume that $\{X_i\}_{i\geq 1}$ and $ \{N(t)\}_{t\geq 0}$ are independent. Then $\{N^{(k)}(t)\}_{t\geq 0}$, defined by
$ N^{(k)}(t) = \sum_{i=1}^{N(t)} X_i$
is called the Poisson process of order $k$ (PPoK) and is denoted by $PPoK(\lambda)$.
\end{definition}

\noindent 
However, the clumping behavior associated with random phenomenon cannot be handled by PPoK \cite{Poiss-order-k}. Hence, there is a need to generalize this notion as well. We propose the following generalization of PPoK.

\begin{definition}\label{def:cppok}
Let $\{N^{(k)}(t)\}_{t\geq 0}$ be the $PPoK(\lambda)$ and $ \{Y_i\}_{i \ge 1}$ be a sequence of IID random variables, independent of $N^{(k)}(t)$, with cumulative distribution function (CDF) $H$. Then the process $\{Z(t)\}_{t\geq 0}$ defined by $Z(t) = \sum_{i=1}^{N^{(k)}(t)} Y_i$ is called the compound Poisson process of order $k$ (CPPoK) and is denoted by $CPPoK(\lambda, H)$.
\end{definition}

\noindent
From the definition, it is clear that:
\begin{enumerate}
\item[$(i)$] for $k=1$, $\{Z(t)\}_{t \ge 0}$ is $CPP(\lambda, H)$ the usual compound Poisson process.
\item[$(ii)$] for $H=\delta_1$, the Dirac measure at 1, $\{Z(t)\}_{t \ge 0}$ is $PPoK(\lambda)$.
\item[$(iii)$] for $k=1$ and $H = \delta_1$, $\{Z(t)\}_{t \ge 0}$ is $PP(\lambda)$.
\end{enumerate}
\noindent
Next, we present a characterization of $CPPoK(\lambda, H)$, in terms of, the finite dimensional distribution (FDD).

\begin{theorem}\label{thm:fdd}
Let $\{Z(t)\}_{t \ge 0}$ be as defined in Definition \ref{def:cppok}. Then the FDD, denoted as $F_{Z(t_1),\ldots,Z(t_n)}(y_1,\ldots,y_n)=\mathbb{P}[Z(t_1)\leq y_1,\ldots, Z(t_n)\leq y_n ]$ has the following form
\begin{equation} \label{eq:FDD}         
F_{Z(t_1),\ldots,Z(t_n)}(y_1,\ldots,y_n)=
\sum_{j_1,\ldots j_n}\prod_{l=1}^{n} p_{j_l}(\Delta t_l)\int_{-\infty}^{v_1} \ldots \int_{-\infty}^{v_n}
\prod_{m=1}^{n}h^{*j_m}_{Y_1}(x_m) dx_m,
\end{equation}
where the summation is taken over all non-negative integers $j_i \geq 0,~ i=1,2,\ldots,n$, $v_k = y_k -\sum_{l=1}^{k-1}x_l, k=1,\ldots,n$, $h$ is the density/pmf of $H$, and $(*j_n)$ represents the $j_n$-fold convolution.
\end{theorem}

\begin{proof}
Let $0=t_0\leq t_1\leq \ldots \leq t_n=t$ be the partition of $[0,t]$. Since, the increments of $\{N^{(k)}(t)\}$ are independent and stationary, we can write $N^{(k)}(t_i) = \sum_{l=1}^{i}N^{(k)}(\Delta t_l),i=1,\ldots,n,$ and $\mathbb{P}[N^{(k)}(t)=j] = p_{j}(t), j=0,1,\ldots$.
\begin{align*}
F_{Z(t_1),\ldots,Z(t_n)}(y_1,\ldots,y_n)=&  \mathbb{P}[\sum_{i=1}^{N^{(k)}(t_1)} Y_i\leq y_1,
		 \ldots, \sum_{i=1}^{N^{(k)}(t_n)} Y_i\leq y_n ]\\
		=& \sum_{j_1,\ldots j_n}\mathbb{P}[\sum_{i=1}^{j_1} Y_i\leq y_1,
		\ldots, \sum_{i=1}^{\sum_{l=1}^{n}j_l} Y_i\leq y_n ] \prod_{l=1}^{n}\mathbb{P}[N^{(k)}(\Delta t_l)=j_l]
\end{align*}	
Let us denote $\sum_{i=1}^{j_1} Y_i=Y(j_1),\ldots, \sum_{i=j_1+\ldots +j_{n-1}+1}^{j_1+\ldots +j_n} Y_i=Y(j_n)$, then it becomes 
\begin{align*}
F_{Z(t_1),\ldots,Z(t_n)}(y_1,\ldots,y_n)	=& \sum_{j_1,\ldots j_n}\prod_{l=1}^{n} p_{j_l}(\Delta t_l)\mathbb{P}[Y(j_1)\leq y_1,
		 \ldots,\sum_{l=1}^{n}Y(j_l)\leq y_n] \\
		=& \sum_{j_1,\ldots j_n}\prod_{l=1}^{n} p_{j_l}(\Delta t_l)\int_{-\infty}^{y_1}\int_{-\infty}^{y_2-x_1} \ldots \int_{-\infty}^{y_{n}-(\sum_{l=1}^{n-1} x_l)}
		\prod_{m=1}^{n}h^{*j_m}_{Y_1}(x_m) dx_m \\
		=& \sum_{j_1,\ldots j_n}\prod_{l=1}^{n} p_{j_l}(\Delta t_l)\int_{-\infty}^{v_1} \ldots \int_{-\infty}^{v_n}
		\prod_{m=1}^{n}h^{*j_m}_{Y_1}(x_m) dx_m.
		\end{align*}
		\ifx
		\begin{align*}
		\mathbb{P}[Z(t_1)\leq y_1,Z(t_2)\leq y_2,\ldots, Z(t_n)\leq y_n ] =& \mathbb{P}[\sum_{i=1}^{N^{(k)}(t_1)} Y_i\leq y_1,\sum_{i=1}^{N^{(k)}(t_2)} Y_i\leq y_2,\ldots, \sum_{i=1}^{N^{(k)}(t_n)} Y_i\leq y_n ]\\
		=& \mathbb{P}[\sum_{i=1}^{N^{(k)}(t_1)} Y_i\leq y_1,\sum_{i=1}^{N^{(k)}(t_1)+N^{(k)}(t_2-t_1)} Y_i\leq y_2,\ldots, \\ 
		& \sum_{i=1}^{N^{(k)}(t_1)+N^{(k)}(t_2-t_{1})+\ldots +N^{(k)}(t_n-t_{n-1})} Y_i\leq y_n ]\\
		=& \sum_{j_1,j_2,\ldots j_n}\mathbb{P}[\sum_{i=1}^{N^{(k)}(t_1)} Y_i\leq y_1,\sum_{i=1}^{N^{(k)}(t_1)+N^{(k)}(t_2-t_1)} Y_i\leq y_2,\ldots, \\ 
		& \sum_{i=1}^{N^{(k)}(t_1)+\ldots +N^{(k)}(t_n-t_{n-1})} Y_i\leq y_n |N^{(k)}(t_1)=j_1,\\
		& N^{(k)}(t_2-t_1)=j_2,\ldots,N^{(k)}(t_n-t_{n-1})=j_n ]\\
		\end{align*}
		\fi 
	\end{proof}
	\begin{remark}
		For $n=1$, \eqref{eq:FDD} reduces to 
		$\mathbb{P}[Z(t)\leq y_1] =\sum_{j_1=0}^{\infty}p_j(t_1)\int_{-\infty}^{y_1}h^{*j_1}_{Y_1}(x_1)dx_1,$ 
		which is the marginal distribution of CPPoK($\lambda$,H).
	\end{remark}
	\ifx
	\begin{theorem}
		Let $G^{n}(y) = \mathbb{P}[Y_1+Y_2+\ldots+Y_n \leq y] $ be the convolution function of n iid random variables. Then, the distribution function of the process $\{Z(t)\}_{t\geq 0}$ can be given as
		$$	\mathbb{P}[Z(t)\leq z] = e^{-k\lambda t}\sum_{n=0}^{\infty}G^n(z) \sum_{\substack{x_1,x_2,...,x_k\geq 0\\\sum_{j=1}^{k}jx_j=n}} \frac{  (\lambda t) ^{x_1+x_2+..+x_k}}{x_1!x_2!...x_k!},$$
		where, the summation is taken over all non-negative integers $ x_1,x_2,\ldots,x_k$ such that $ x_1+2x_2+\ldots +kx_k = n.$
	\end{theorem}
	
	\begin{proof}
		\begin{eqnarray*}
			\mathbb{P}[Z(t)\leq z] &=& \mathbb{P}[\sum_{i=1}^{N^{(k)}(t)} Y_i \leq z] \\&=&
			\sum_{n=0}^{\infty} \mathbb{P}[\sum_{i=1}^{N^{(k)}(t)} Y_i \leq z |N_{k}(t) = n]\mathbb{P}[N^{(k)}(t)=n] \\ &=& \sum_{n=0}^{\infty} \mathbb{P}[\sum_{i=1}^n Y_i \leq z]\mathbb{P}[N^{(k)}(t)=n] \\ &=& \sum_{n=0}^{\infty} G^n(z)\mathbb{P}[N^{(k)}(t)=n]\\ &=& e^{-k\lambda t}\sum_{n=0}^{\infty}G^n(z) \sum_{\substack{x_1,x_2,...,x_k\geq 0\\\sum_{j=1}^{k}jx_j=n}} \frac{  (\lambda t) ^{x_1+x_2+..+x_k}}{x_1!x_2!...x_k!}.
		\end{eqnarray*}
	\end{proof}
	
	\noindent
	Now we consider $Y_1,Y_2,\ldots$ are non-negative integer valued random variable. Then its probability mass function is given as,
	\begin{eqnarray*}
		\mathbb{P}(Z(t) = m) &=&
		\sum_{n=0}^{\infty} f^{*n}(m)\mathbb{P}(N^{(k)}(t) = n),~~m=0,1,2,\ldots, 
	\end{eqnarray*}
	where $ f^{*n}(m)= \mathbb{P}(Y_1+Y_2+..Y_n = m) $ is the $n$-th fold convolution of $ Y_1,Y_2,\ldots Y_n$.
	\vskip 2ex
	\noindent
	Panjer (\cite{panjer_1981}) gave an recursive algorithm to calculate the probability distribution of compound random variable. Then many authors generalize it for many compound distributions. Now, we give the recursive scheme to calculate the distribution of CPPoK. \\
	Let $p_n = \mathbb{P}[N^{(k)} = n]$ be the \textit{pmf} of PPoK, then after some algebraic calculation, we get the following relation 
	\begin{equation}\label{recursive}
	p_n = \sum_{i=1}^{k}\frac{b_i}{n}p_{n-i},~~~n=1,2,\ldots,
	\end{equation}
	where, $b_i = i\lambda.$\\
	So the distribution of Poison distribution of order $k$ is  $\mathbf{R}_k[\mathbf{a,b}]$ as defined in see(\cite{panjer}) where $\mathbf{a}=(0,0,\ldots,0)_{k-tuple}$ and $\mathbf{b} =(b_1,b_2,\ldots,b_k)$.
	\begin{proposition}
		Let $ Z(t) = \sum_{i=1}^{N^{(k)}(t)} Y_i$ be the CPPoK where $\{Y_i\}_{i\geq 1}$ be the sequence of IID non-negative integer valued random variable with density $f$ and $ N^{(k)}$ be the PPoK whose distribution  the distribution is $\mathbf{R}_k[\mathbf{0,b}]$ as defined above. Then the \textit{pmf} of  $\{Z(t)\}_{t\geq 0}$ can be calculated by the formula,
		$$\mathbb{P}(Z(t) = m) = \sum_{h=1}^{m}\mathbb{P}(Z(t) = m-h)\sum_{i=1}^{k}\frac{\lambda h}{m}f^{*i}(h),~~m=1,2,\ldots. $$
		and  $$\mathbb{P}(Z(t) = 0) = \sum_{n=0}^{\infty} (f(0))^n p_n.$$
		where $ p_n = \mathbb{P}(N^{(k)} = n)$.
	\end{proposition}
	\begin{proof}
		By using \eqref{recursive} in \ref[Theorem 9]{panjer}, we get this recursive relation.
	\end{proof}
	\fi
	\noindent 
	The mean and variance of the process $\{Z(t)\}_{t\geq 0}$ can be expressed as
	\begin{align}
	 \mathbb{E}[Z(t)] &= \frac{k(k+1)}{2}\lambda t \mathbb{E}[Y],\\
	\text{Var}[Z(t)] &= \frac{k(k+1)}{2}\lambda t \text{Var}(Y) + \frac{k(k+1)(2k+1)}{6}\lambda t \mathbb{E}[Y]^2. 
	\end{align}
	\subsection{Index of dispersion} In this subsection, we discuss the index of dispersion of $CPPoK(\lambda,H)$.
	\begin{definition}[Maheshwari and Vellaisamy \cite{TCFPP-pub}] \label{def:iod}
		The index of dispersion for a counting process $\{Z(t)\}_{t\geq 0}$ is defined by
		$$ I(t) = \frac{\text{Var}[Z(t)]}{\mathbb{E}[Z(t)]}.$$
		Then the stochastic process $\{Z(t)\}_{t\geq 0}$ is said to be overdispersed if $I(t)>1$, underdispersed if $I(t)<1$, and equidispersed if $ I(t) = 1$.
	\end{definition}
	\noindent
	Alternatively, Definition \ref{def:iod} can be interpreted as follows. A stochastic process $\{Z(t)\}_{t\geq 0}$ is over(under)-dispersed if $\text{Var}[Z(t)]-\mathbb{E}[Z(t)]>0(<0)$. Therefore, we first calculate
	
	\begin{eqnarray} \label{eq:iod}
		\text{Var}[Z(t)]-\mathbb{E}[Z(t)] &=& \frac{k(k+1)}{2}\lambda t \left[ \text{Var}(Y) - \mathbb{E}(Y) + \frac{(2k+1)}{3} \mathbb{E}[Y]^2\right] \nonumber \\ &=& \frac{k(k+1)}{2}\lambda t \left[ \mathbb{E}(Y^2)-\mathbb{E}(Y) + \left( \frac{2k+1}{3}-1\right) \mathbb{E}(Y)^2 \right].
	\end{eqnarray}
	From the above definition, the following cases arise:
	\begin{enumerate}
		\item[$(i)$]  If $Y_i's$ are over and equidispersed, then $CPPoK(\lambda,H)$ exhibits overdispersion.
		\item[$(ii)$] If $Y_i's$ are underdispersed with non-negative integer valued random variable i.e., $[\mathbb{E}(Y^2)-\mathbb{E}(Y)]\geq 0$, then $CPPoK(\lambda,H)$ shows overdispersion.
		\item[$(iii)$]  If $Y_i's$ are underdispersed, then  $CPPoK(\lambda,H)$ may show both over and underdispersion.
		e.g. $Y_i \sim exp(\mu)$, \eqref{eq:iod} becomes 
		$$ \text{Var}[Z(t)]-\mathbb{E}[Z(t)] = \frac{k(k+1)}{2}\lambda t\left[ \frac{1}{\mu^2} \left(\frac{2k+4}{3}\right)-\frac{1}{\mu}\right]. $$
		\noindent
		If $\mu < \frac{2k+4}{3}$, $CPPoK(\lambda,H)$ exhibits overdispersion  else underdispersion. Hence, unlike $PPoK(\lambda)$, it can capture both the cases.
	\end{enumerate}


	\vskip 2ex
	\subsection{Long range dependence}
	In this subsection, we prove the long-range dependence (LRD) property for the $CPPoK(\lambda,H)$. There are several definitions available in literature. We used the definition given in \cite{lrd2016}.
	\begin{definition}[Maheshwari and Vellaisamy \cite{lrd2016}] \label{Def:LRD} 
		Let $ 0\leq s<t$ and $s$ be fixed. Assume a stochastic process $ \{ X(t)\}_{t \geq 0}$ has the correlation function $ Corr[X(s),X(t)]$ that satisfies
		$$ c_1(s)t^{-d} \leq Corr[X(s),X(t)] \leq c_2(s)t^{-d}, $$
		for large $ t,d >0,~ c_1(s)>0 ~and~ c_2(s)>0$. i.e.,
		$$ \lim_{t \rightarrow \infty} \frac{Corr[X(s),X(t)]}{t^{-d}} = c(s)$$
		for some $ c(s)>0$ and $ d>0$. We say that, $ X(t)$ has the long-range dependence (LRD) property if $d \in (0,1)$ and short-range dependence (SRD) property if $ d \in (1,2)$.
	\end{definition}
	\begin{proposition}
		The $CPPoK(\lambda,H)$ has the LRD property.
	\end{proposition}
	\begin{proof}
		Let $0\leq s<t$. Consider
		\begin{eqnarray*}
			Corr[Z(s),Z(t)] &=& \frac{Cov[Z(s),Z(t)]}{\sqrt{\text{Var}[Z(s)]\text{Var}[Z(t)]}},\\
			&=& \frac{c(s)}{t^{1/2}\sqrt{\frac{k(k+1)}{2}\lambda \text{Var}(Y) + \frac{k(k+1)(2k+1)}{6}\lambda \mathbb{E}[Y]^2}}\\
			&=& c(s) t^{-1/2},
		\end{eqnarray*}
		where, $ 0< c(s) = \sqrt{\frac{\frac{k(k+1)}{2}\lambda s \text{Var}(Y) + \frac{k(k+1)(2k+1)}{6}\lambda s \mathbb{E}[Y]^2}{\frac{k(k+1)}{2}\lambda \text{Var}(Y) + \frac{k(k+1)(2k+1)}{6}\lambda \mathbb{E}[Y]^2}},$\\
		which decays like the power law $t^{-1/2}$. Hence $CPPoK(\lambda,H)$ has LRD property.	
	\end{proof}
	
	\subsection{Martingale characterization for CPPoK}
	It is well known that the martingale characterization for homogeneous Poisson process is called Watanabe theorem (see \cite{point}). Now, we extend this theorem for $CPPoK(\lambda,H)$, where $H$ is discrete distribution with support on $\mathbb{Z^+}$ and for this we need  following two lemmas.
	\begin{lemma} \label{martingale}
		Let $ D(t) = \sum_{j=1}^{N(t)} X_j ,~t\geq 0 $ is the compound Poisson process, where $\{N(t)\}_{t\geq 0}$ is $PP(k\lambda)$ and $\{X_j\}_{j\geq 1}$ are non-negative  \textit{IID} random variable, independent from $\{N(t)\}_{t\geq 0}$, with \textit{pmf}  $\mathbb{P}(X_j=i) = \alpha_i,~ (i=0,1,2,\ldots,j=1,2,\ldots)$. Then $\{D(t)\}_{t\geq 0}$ can be represented as
		$$ D(t) \stackrel{d}{=} \sum_{i=1}^{\infty} iZ_i(t),~~t\geq 0,$$
		where, $Z_i(t),~i=1,2,\ldots$ are independent $PP(k\lambda \alpha_i)$, and the symbol $\stackrel{d}{=}$ denotes the equality in distribution.
	\end{lemma}
	\begin{proof} 
		Let $G_{D(t)}(u)$ is the probability generating function (\textit{pgf}) of $\{D(t)\}_{t\geq 0}$, then
    $$G_{D(t)}(u) = \mathbb{E}[u^{D(t)}] =
		 \exp[\lambda kt \sum_{j=1}^{\infty}\alpha_j(u^j -1)] = \mathbb{E}[u^{\sum_{i=1}^{\infty} iZ_i(t)}].$$
	\end{proof}
	\begin{lemma}
		The \textit{pgf} of $ Z(t) = \sum_{i=1}^{N^{(k)}(t)}Y_i,~t\geq 0 $ has the following form
		$$ G_{Z(t)}(u)  = \exp \left[ \lambda kt \sum_{j_1=1}^{\infty}\frac{q^{(1)}_{j_1} +q_{j_1}^{(2)}+\ldots + q_{j_1}^{(k)}}{k} (u^{j_1}-1) \right] ,$$
		where $Y_i,~i=1,2,\ldots$ are non-negative integer valued IID random variables with $\mathbb{P}[Y_i=n] = q_n,~ n=0,1,\ldots.$
	\end{lemma}
\begin{proof}
	Let $ G_{Z(t)}(u) $ is the \textit{pgf} of $ \{Z(t)\}_{t\geq 0}$, then 
	\begin{eqnarray*}
		G_{Z(t)}(u) = \mathbb{E}[u^{Z(t)}] 
		&=& \exp[\lambda t\{G_Y(u)+\ldots + G_Y^{*k}(u)\} - k\lambda t]\\
		&=& \exp \left[ \lambda t \left\{ \sum_{j_1=0}^{\infty} q_{j_1} u^{j_1} +  \ldots + \sum_{j_1=0}^{\infty} q^{*k}_{j_1} u^{j_1} \right\} -k\lambda t\right]\\
		&=& \exp \left[\lambda t \left\{ \sum_{j_1=0}^{\infty} q_{j_1} u^{j_1} +
		\ldots + 
		 \sum_{j_1=0}^{\infty}\sum_{j_2=0}^{j_1} \ldots \sum_{j_k=0}^{j_{k-1}}q_{j_k}q_{j_{k-1}-j_k}\ldots q_{j_1-j_2}u^{j_1} \right\}-k\lambda t \right] \\
		&& \text{Let us denote}~ q_{j_1}^{(n)} = \prod_{m=1}^{n} \sum_{j_m=0}^{j_{m-1}}q_{j_n}q_{j_{n-1}-j_n}\ldots q_{j_1-j_2},~n=1,\ldots,k. \\
	&=&\exp \left[\lambda kt \sum_{j_1=1}^{\infty}\frac{q_{j_1}^{(1)} +\ldots + q_{j_1}^{(k)}}{k} (u^{j_1}-1) \right].
	\end{eqnarray*}
\end{proof}
\begin{remark}\label{lemma:martingale}
	Set $ \alpha_{j_1} = \frac{q_{j_1}^{(1)} +\ldots + q_{j_1}^{(k)}}{k}, ~j_1=0,1,2,\ldots .$ Substituting $\alpha_{j_1}$ in Lemma \ref{martingale}, we get the following relation
	\begin{equation}\label{distribution}
	D(t) \stackrel{d}{=} \sum_{i=1}^{\infty} iZ_i(t) \stackrel{d}{=} \sum_{i=1}^{N^{(k)}(t)}Y_i,~ t\geq 0,
	\end{equation}
	where $Y_i's$ are non-negative integer valued random variables and $ N^{(k)}(t)$ is $PPoK(\lambda)$. 
\end{remark}
	
	\ifx
	\begin{eqnarray*}
		G_{Z(t)}(u) &=& \mathbb{E}[u^{\sum_{i=1}^{N_k(t)}Y_i}] \\
		&=& \exp[\lambda t\{G_Y(u)+G_Y^{*2}(u)+\ldots + G_Y^{*k}(g)\} - k\lambda t]\\
		&=& \exp \left[ \lambda t \left\{ \sum_{j=1}^{\infty} p_j z^j + \left( \sum_{j=1}^{\infty} p_j z^j\right)^2 \ldots \left( \sum_{j=1}^{\infty} p_j z^j \right)^k \right\} -k\lambda t\right]\\
		&=& \exp \left[\lambda t \left\{ \sum_{j=1}^{\infty} p_j z^j + 
		 \ldots + \sum_{j=1}^{\infty}\sum_{j_2=1}^{j} \ldots \sum_{j_k=1}^{j_{k-1}}p_{j_k}p_{j_{k-1}-j_k}\ldots p_{j-j_2}z^j \right\}-k\lambda t \right] \\
		 & & \text{where}~ p_j^k = \sum_{j_2=1}^{j} \ldots \sum_{j_k=1}^{j_{k-1}}p_{j_k}p_{j_{k-1}-j_k}\ldots p_{j-j_2} \\
		&=& \exp \left[ \lambda t \sum_{j=1}^{\infty} p_j (z^j-1) +
		  \ldots + \lambda t\sum_{j=1}^{\infty} p_{j}^k (z^j-1) \right] \\
		&=& \exp \left[ \lambda t \sum_{j=1}^{\infty}(p_j +p_{j}^2+\ldots + p_{j}^k) (z^j-1) \right] \\
		&=& \exp \left[ \lambda kt \sum_{j=1}^{\infty}\frac{(p_j +p_{j}^2+\ldots + p_{j}^k)}{k} (z^j-1) \right].
	\end{eqnarray*}
	\begin{equation}\label{alpha}
	Set~ \alpha_j = \frac{(p_j +p_{j}^2+\ldots + p_{j}^k)}{k}, ~j=1,2,\ldots.
	\end{equation}
	\fi
	
\begin{theorem}
	Let $\{Z(t)\}_{t\geq 0}$ be $ 
	CPPoK(\lambda,H)$, where $H$ is discrete distribution with support on $\mathbb{Z^+}$, is a $\mathcal{F}_t$ adapted stochastic process \textit{iff} $ M(t) = Z(t) - \frac{k(k+1)}{2}\lambda t \mathbb{E}[Y],~ t\geq 0$ is an $\mathcal{F}_t$ martingale.
\end{theorem}
	\begin{proof}
		Let $Z(t)$ be the $\mathcal{F}_t$ adapted stochastic process. If $Z(t)$ is a compound Poisson process of order $k$, then 
		\begin{eqnarray*}
			\mathbb{E}[M(t)|\mathcal{F}_s] &=& \mathbb{E}[ Z(t) - \frac{k(k+1)}{2}\lambda t \mathbb{E}[Y]|\mathcal{F}_s],~~0\leq s \leq t.\\
			&=& \mathbb{E}[ Z(t)|\mathcal{F}_s] - \frac{k(k+1)}{2}\lambda t \mathbb{E}[Y]\\
			&=& \mathbb{E}[ Z(t)-Z(s)|\mathcal{F}_s]+ \mathbb{E}[ Z(s)|\mathcal{F}_s]- \frac{k(k+1)}{2}\lambda t \mathbb{E}[Y]\\
			&=& \mathbb{E}[ Z(t)-Z(s)]+ Z(s)- \frac{k(k+1)}{2}\lambda t \mathbb{E}[Y]\\
			&=&  Z(s)- \frac{k(k+1)}{2}\lambda s \mathbb{E}[Y] = M(s).
		\end{eqnarray*}
	Hence, the process $\{M(t)\}_{t\geq 0}$ is $\mathcal{F}_t$ martingale.\\
   Since, in Remark \ref{lemma:martingale}, it is shown that 
$ \sum_{i=1}^{\infty} iZ_i(t) \stackrel{d}{=} \sum_{i=1}^{N^{(k)}(t)}Y_i.$ So the other part easily follows using \cite[Theorem 5.2]{discrete}.
	\end{proof}
	\begin{remark}
		We know that CPP is a L\'evy process and in \eqref{distribution}, it is proved that CPPoK is equal in distribution to $\{D(t)\}_{t\geq 0}$. Hence, CPPoK is also a L\'evy process and hence infinitely divisible.
	\end{remark}
	\begin{remark}
		The characteristic function of $CPPoK(\lambda,H)$ can be written as 
		$$ \mathbb{E}[e^{iwZ(t)}] = \exp[t\sum_{j=1}^{\infty}(e^{iwj}-1) k\lambda \alpha_j],$$
		where, $\alpha_j, j=1,2,\ldots$ are as defined in Remark \ref{lemma:martingale}, and $ k\lambda \alpha_j = \nu_j$ is called the L\'evy measure of $CPPoK(\lambda,H)$.
	\end{remark}
	\vskip 2ex
	\noindent
	\ifx
	We can also prove the infinite divisibility with the help of characteristic function of CPPoK.
	\begin{proposition}
		The process $\{ Z(t)\}_{t\geq 0}$ is infinitely divisible.
		\begin{proof}
			The characteristic function of $Z(t) $ is given by
			\begin{eqnarray*}
				\psi_{Z(t)}(w) &=& \mathbb{E}[e^{iwZ(t)}]\\
				&=& \sum_{n=0}^{\infty} (\psi_Y(w))^n \mathbb{P}[N_k(t)=n]\\
				&=&  \sum_{n=0}^{\infty} (\psi_Y(w))^n \sum_{\substack{x_1,x_2,...,x_k\geq 0\\\sum_{j=1}^{k}jx_j=n}} e^{-k\lambda t} \frac{  (\lambda t) ^{x_1+x_2+..+x_k}}{x_1!x_2!...x_k!} \\
				&=& \exp[-k\lambda t + \lambda t(\psi_Y(w)+\psi^2_Y(w)+\ldots+\psi^k_Y(w))]\\
				&=& \exp[-k\frac{\lambda}{n} t + \frac{\lambda}{n} t(\psi_Y(w)+\psi^2_Y(w)+\ldots+\psi^k_Y(w))]^n\\ 
				&=& [\psi_{V(t)}(w)]^n
			\end{eqnarray*}
			Where $V(t)$ has same distribution as $Z(t)$ with parameter $\lambda/n$. Thus it is proved.
		\end{proof}
	\end{proposition}
	\begin{remark}
		The Laplace Transform (LT) of $Z(t)$ can be expressed as
		\begin{eqnarray} \label{laplace}
		\phi_{Z(t)}(s) = \mathbb{E}[e^{-sZ(t)}]= \exp[-k\lambda t + \lambda t(f(s)+f^2(s)+ \ldots +f^k(s))],
		\end{eqnarray}
		
		where $f(s)$ is the LT of random variable $Y$ and $f^k(s) = (\mathbb{E}[s^{-sY}])^k$.
	\end{remark}
	We can rewrite the LT as
	\begin{eqnarray*}
		\phi_{Z(t)}(s) &=& \exp[-k\lambda t + \lambda t(f(s)+f^2(s)+ \ldots +f^k(s))]\\ &=& \exp[-\lambda t(1-f(s))]\exp[-\lambda t(1-f^2(s))]\ldots \exp[-\lambda t(1-f^k(s))]
	\end{eqnarray*}
	Now, $\exp[-\lambda t(1-f^k(s))]$ is the LT of the compound Poisson process $\sum_{i=1}^{kN_k(t)}Y_i$.\\
	Thus the process $Z(t)$ can be rewritten as 
	$$ Z(t) \stackrel{d}{=} \sum_{i=1}^{N_1(t)}Y_i+\sum_{i=1}^{2N_2(t)}Y_i+\ldots +\sum_{i=1}^{kN_k(t)}Y_i. $$
	\subsection{L\'evy Subordinator} \label{levy}
	A subordinator  $\{D_f(t)\}_{t\geq0}$ is a one-dimensional non-decreasing L\'evy process whose LT can be expressed in the form (see \cite{bertoin})
	$$ \mathbb{E}[e^{-\lambda D_f(t)}] = e^{-tf(\lambda)},~ \lambda>0,$$
	where the function $ f:[0,\infty) \rightarrow [0,\infty) $ is called the Laplace exponent and 
	$$ f(\lambda)=b \lambda+\int_{0}^{\infty}(1-e^{-\lambda x})\nu(dx),~b\geq0, s>0,$$
	Here $b$ is the drift coefficient and $\nu$ is a non-negative L\'evy measure on positive half-line satisfying 
	\begin{equation*}
	\int_{0}^{\infty}(x\wedge 1)\nu(dx)<\infty~~{\rm and}~~\nu([0,\infty))= \infty
	\end{equation*}
	which ensures that the sample paths of $D_{f}(t)$ are almost surely $(a.s.)$  strictly increasing.
	Also the inverse subordinator $\{E_f(t)\}_{t\geq 0}$ is the first exit time of the L\'evy suboridinator $\{D_f(t)\}_{t\geq 0}$, which is defined as 
	\begin{equation*}
	E_{f}(t)=\inf\{r\geq 0:D_{f}(r)>t\},~ t\geq 0.
	\end{equation*}
	\fi
	
	
	\section{Main Results} \label{sec:subordinator}
	\noindent
	In this section, we recall the definitions of L\'evy subordinator and its first exit time. Further, we define the subordinated versions of $CPPoK(\lambda,H)$ 
	and discuss their properties.
	\subsection*{L\'evy Subordinator}\label{Levy}
	A L\'evy subordinator  $\{D_f(t)\}_{t\geq0}$ is a one-dimensional non-decreasing L\'evy process whose Laplace transform (LT) can be expressed in the form (see \cite{bertoin})
	$$ \mathbb{E}[e^{-\lambda D_f(t)}] = e^{-tf(\lambda)},~ \lambda>0,$$
	where the function $ f:[0,\infty) \rightarrow [0,\infty) $ is called the Laplace exponent and 
	$$ f(\lambda)=b \lambda+\int_{0}^{\infty}(1-e^{-\lambda x})\nu(dx),~b\geq0, s>0.$$
	Here $b$ is the drift coefficient and $\nu$ is a non-negative L\'evy measure on positive half-line satisfying 
	\begin{equation*}
	\int_{0}^{\infty}(x\wedge 1)\nu(dx)<\infty~~{\rm and}~~\nu([0,\infty))= \infty,
	\end{equation*}
	which ensures that the sample paths of $D_{f}(t)$ are almost surely $(a.s.)$  strictly increasing.
	Also, the inverse subordinator $\{E_f(t)\}_{t\geq 0}$ is the first exit time of the L\'evy subordinator $\{D_f(t)\}_{t\geq 0}$, and it is defined as 
	\begin{equation*}
	E_{f}(t)=\inf\{r\geq 0:D_{f}(r)>t\},~ t\geq 0.
	\end{equation*}
	\noindent
	Next, we study $CPPoK(\lambda,H)$ by taking subordinator as mixture of tempered stable subordinators (MTSS).
	\subsection{CPPoK time changed by mixtures of tempered stable subordinators}
	The mixtures of tempered stable subordinator (MTSS) $ \{S_{\alpha_1,\alpha_2}^{\mu_1,\mu_2}(t)\}_{t\geq 0}$ is a L\'evy process with LT (see \cite{mixedsubordinator})
	$$ \mathbb{E}[e^{-sS_{\alpha_1,\alpha_2}^{\mu_1,\mu_2}(t)}] = \exp\{-t(c_1((s+\mu_1)^{\alpha_1} - \mu_1^{\alpha_1})+c_2((s+\mu_2)^{\alpha_2} - \mu_2^{\alpha_2}))\} ,~ s>0,$$
	
	\noindent where $c_1+c_2=1, c_1,c_2 \geq 0, \mu_1,\mu_2>0 $ are tempering parameters and $ \alpha_1, \alpha_2 \in (0,1)$ are stability indices. The function $f(s) = c_1((s+\mu_1)^{\alpha_1} - \mu_1^{\alpha_1})+c_2((s+\mu_2)^{\alpha_2} - \mu_2^{\alpha_2}) $ is the Laplace exponent of MTSS.\\
	The mean and variance of MTSS are given as
	\begin{equation}\label{eq:mean}
	\mathbb{E}[S_{\alpha_1,\alpha_2}^{\mu_1,\mu_2}(t)] = t(c_1\alpha_1\mu_1^{\alpha_1-1}+c_2\alpha_2\mu_2^{\alpha_2-1}),
	\end{equation}
	\begin{equation}\label{eq:variance}
	\text{Var}[S_{\alpha_1,\alpha_2}^{\mu_1,\mu_2}(t)]= t(c_1\alpha_1(1-\alpha_1)\mu_1^{\alpha_1-2}+c_2\alpha_2(1-\alpha_2)\mu_2^{\alpha_2-2}).
	\end{equation}
	\begin{definition}
		Let $\{S_{\alpha_1,\alpha_2}^{\mu_1,\mu_2}(t)\}_{t\geq 0}$ be the L\'evy subordinator satisfying $\mathbb{E}[S_{\alpha_1,\alpha_2}^{\mu_1,\mu_2}(t)^{\rho}]< \infty$ for all $\rho >0.$ Then the time-changed $CPPoK(\lambda,H)$, denoted by $TCPPoK(\lambda,H,S_{\alpha_1,\alpha_2}^{\mu_1,\mu_2})$ is defined as
		$$ Z_1(t) = Z(S_{\alpha_1,\alpha_2}^{\mu_1,\mu_2}(t)) = \sum_{i=1}^{N^{(k)}(S_{\alpha_1,\alpha_2}^{\mu_1,\mu_2}(t))} Y_i ,~ t \geq 0,$$
		where $\{Z(t)\}_{t\geq 0}$ is $CPPoK(\lambda,H)$, independent from $ \{S_{\alpha_1,\alpha_2}^{\mu_1,\mu_2}(t)\}_{t\geq 0} $.
	\end{definition}
	
	\begin{remark}
		If $\alpha_1 = \alpha_2=\alpha,$ and $ \mu_1=\mu_2 = 0$, then MTSS becomes $\alpha$-stable subordinator, reducing $Z_1(t)$ to $TCPPoK(\lambda,H,S_{\alpha})$, which we call as space fractional CPPoK and written as
		$$ Z(S_{\alpha}(t)) = \sum_{i=1}^{N^{(k)}(S_{\alpha}(t))} Y_i ,~ t \geq 0. $$
		When $Y_i's \equiv 1$, then $Z(S_{\alpha}(t))$ reduces to space fractional PPoK, denoted as $PPoK(\lambda,S_{\alpha})$.
		It can be seen as extension of space fractional Poisson process (see \cite{sfpp}). 
	\end{remark}
	\begin{remark}
		If $\alpha_1 = \alpha_2=\alpha,$ and $ \mu_1= \mu,\mu_2 = 0$, then MTSS reduces to  tempered $\alpha$- stable subordinator and  $Z_1(t)$ becomes tempered space fractional CPPoK, denoted as $TCPPoK(\lambda,H,S_{\alpha}^{\mu} )$, can be written as 
		$$ Z(S_{\alpha}^{\mu}(t)) = \sum_{i=1}^{N^{(k)}(S_{\alpha}^{\mu}(t))} Y_i ,~ t \geq 0, $$
		where $\mu>0$ is tempering parameter. Substituting $Y_i's \equiv 1$, it becomes tempered space fractional PPoK, denoted as $PPoK(\lambda,S_{\alpha}^{\mu})$.
	\end{remark}

\begin{theorem}
	The finite dimensional distribution of $TCPPoK(\lambda,H,S_{\alpha_1,\alpha_2}^{\mu_1,\mu_2})$ has the following form
	
	\begin{equation}  \label{eq:Fdd}       
	F_{Z_1(t_1),\ldots,Z_1(t_n)}(y_1,\ldots,y_n)=
	\sum_{j_1,\ldots j_n}\prod_{l=1}^{n} q_{j_l}(\Delta t_l)\int_{-\infty}^{v_1} \ldots \int_{-\infty}^{v_n}
	\prod_{m=1}^{n}h^{*j_m}_{Y_1}(x_m) dx_m,
	\end{equation}
   where the summation is taken over all non-negative integers $j_i \geq 0,~ i=1,2,\ldots,n$, $v_k = y_k -\sum_{l=1}^{k-1}x_l, k=1,\ldots,n$, $h$ is the density of $H$, and $q_j(t) = \mathbb{P}[N^{(k)}(D_f(t))=j]$.
\end{theorem}
\begin{proof}
	The result easily follows from the proof of Theorem \ref{thm:fdd}.
\end{proof}
	\ifx
	\begin{proposition} Let $\{Z(S_{\alpha_1,\alpha_2}^{\mu_1,\mu_2}(t))\}_{t\geq 0}$ be the CPPoK time changed with MTSS. Then the mean and covariance function of this process is given as
		\begin{enumerate}[(i)]
			\item $\mathbb{E}[Z(S_{\alpha_1,\alpha_2}^{\mu_1,\mu_2}(t))]= \frac{k(k+1)}{2} \lambda \mathbb{E}[Y] \mathbb{E}[S_{\alpha_1,\alpha_2}^{\mu_1,\mu_2}(t)].$
			\item $ Cov[Z(S_{\alpha_1,\alpha_2}^{\mu_1,\mu_2}(s)),Z(S_{\alpha_1,\alpha_2}^{\mu_1,\mu_2}(t))] = \frac{k(k+1)}{2}\lambda \text{Var}[Y]\mathbb{E}[Z(S_{\alpha_1,\alpha_2}^{\mu_1,\mu_2}(s))] +  \frac{k(k+1)(2k+1)}{6}\lambda \mathbb{E}[Y]^2\\  \mathbb{E}[Z(S_{\alpha_1,\alpha_2}^{\mu_1,\mu_2}(s))] +  ( \frac{k(k+1)}{2} \lambda \mathbb{E}[Y])^2 \text{Var}[Z(S_{\alpha_1,\alpha_2}^{\mu_1,\mu_2}(s))].$
		\end{enumerate}
	\end{proposition}
	\noindent	
	On putting $s=t$ in $(ii)$, we get variance.
	\noindent
	Now, we evaluate $ \text{Var}[Z(S_{\alpha_1,\alpha_2}^{\mu_1,\mu_2}(t))] - \mathbb{E}[Z(S_{\alpha_1,\alpha_2}^{\mu_1,\mu_2}(t))]$,
	\begin{eqnarray*}
		&=&   \frac{k(k+1)}{2} \lambda \mathbb{E}[S_{\alpha_1,\alpha_2}^{\mu_1,\mu_2}(t)]\left[\text{Var}[Y]-\mathbb{E}[Y]+ \frac{2k+1}{3}\mathbb{E}[Y]^2 \right] +
		\\ && \left(\frac{k(k+1)}{2} \lambda \mathbb{E}[Y]\right)^2 \text{Var}[S_{\alpha_1,\alpha_2}^{\mu_1,\mu_2}(t)]\\ &>& 0.
	\end{eqnarray*}
	Hence it exhibits overdispersion.\\
	\fi

	\vskip 2ex
	\noindent
	Now, we present some distributional properties of $TCPPoK(\lambda,H,S_{\alpha_1,\alpha_2}^{\mu_1,\mu_2})$.
	\begin{theorem}\label{thm:dis}
		Let $0<s\leq t<\infty$. Then the mean and covariance function of $TCPPoK(\lambda,H,S_{\alpha_1,\alpha_2}^{\mu_1,\mu_2})$ are given as
		\begin{enumerate}[(i)]
			\item $\mathbb{E}[Z_1(t)]= \frac{k(k+1)}{2} \lambda \mathbb{E}[Y] \mathbb{E}[S_{\alpha_1,\alpha_2}^{\mu_1,\mu_2}(t)]$,
			\item $ Cov[Z_1(s),Z_1(t)] = \mathbb{E}[Z(1)]^2 \text{Var}[S_{\alpha_1,\alpha_2}^{\mu_1,\mu_2}(s)]+\mathbb{E}[S_{\alpha_1,\alpha_2}^{\mu_1,\mu_2}(s)]\text{Var}[Z(1)].$
		\end{enumerate}
	\end{theorem}
	\begin{proof} Let $g_f(y,t)$ be the \textit{pdf} of L\'evy subordinator $\{S_{\alpha_1,\alpha_2}^{\mu_1,\mu_2}(t)\}_{t\geq 0}$. Then
		$$ \mathbb{E}[Z_1(t)]= \mathbb{E}[\mathbb{E}[Z(S_{\alpha_1,\alpha_2}^{\mu_1,\mu_2}(t))|S_{\alpha_1,\alpha_2}^{\mu_1,\mu_2}(t)]] = \frac{k(k+1)}{2} \lambda \mathbb{E}[Y] \mathbb{E}[S_{\alpha_1,\alpha_2}^{\mu_1,\mu_2}(t)]. $$
		Hence, part (i) is proved.\\
		\noindent
		Now by using conditioning argument we can write 
		$$\mathbb{E}[Z_1(s)Z_1(t)] = \mathbb{E}[S_{\alpha_1,\alpha_2}^{\mu_1,\mu_2}(s)S_{\alpha_1,\alpha_2}^{\mu_1,\mu_2}(t)]\mathbb{E}[Z(1)]^2+\mathbb{E}[S_{\alpha_1,\alpha_2}^{\mu_1,\mu_2}(s)]\text{Var}[Z(1)] $$
		Therefore, we get
		\begin{align*}
		\mbox{Cov}[Z_1(s),Z_1(t)] =& \mathbb{E}[Z_1(s)Z_1(t)] - \mathbb{E}[Z_1(s)]\mathbb{E}[Z_1(t)]\\
		=&\mathbb{E}[Z_1(s)Z_1(t)]-\mathbb{E}[Z(1)]^2\mathbb{E}[S_{\alpha_1,\alpha_2}^{\mu_1,\mu_2}(s)]\mathbb{E}[S_{\alpha_1,\alpha_2}^{\mu_1,\mu_2}(t)]\\
		=& \mathbb{E}[Z(1)]^2 \text{Var}[S_{\alpha_1,\alpha_2}^{\mu_1,\mu_2}(s)]+\mathbb{E}[S_{\alpha_1,\alpha_2}^{\mu_1,\mu_2}(s)]\text{Var}[Z(1)].
		\end{align*}
		which completes the proof. On putting $s=t$ in part (ii), we can get the expression for variance of $TCPPoK(\lambda,H,S_{\alpha_1,\alpha_2}^{\mu_1,\mu_2})$.
	\end{proof}
\noindent
	Now, we discuss the index of dispersion for $TCPPoK(\lambda,H,S_{\alpha_1,\alpha_2}^{\mu_1,\mu_2})$. For this, we evaluate 
	\begin{align*}
	\text{Var}[Z_1(t)]-\mathbb{E}[Z_1(t)] =& \mathbb{E}[Z(1)]^2 \text{Var}[S_{\alpha_1,\alpha_2}^{\mu_1,\mu_2}(t)]+\mathbb{E}[S_{\alpha_1,\alpha_2}^{\mu_1,\mu_2}(t)]\left\{\text{Var}[Z(1)]-\mathbb{E}[Z(1)]\right\}.
	\end{align*}
	Since $\{S_{\alpha_1,\alpha_2}^{\mu_1,\mu_2}(t)\}_{t\geq 0}$ is a L\'evy subordinator, therefore $\mathbb{E}[S_{\alpha_1,\alpha_2}^{\mu_1,\mu_2}(t)]>0$. Thus the following cases arises:
	\begin{enumerate}
		\item [$(i)$] If $Z(1)$ is over/equidispersed, then $TCPPoK(\lambda,H,S_{\alpha_1,\alpha_2}^{\mu_1,\mu_2})$ exhibits overdispersion.
			\item [$(ii)$] If $Z(1)$ is underdispersed then
			\begin{enumerate}
				\item [$(a)$] if $\mathbb{E}[Z(1)]^2 \text{Var}[S_{\alpha_1,\alpha_2}^{\mu_1,\mu_2}(t)]> \mathbb{E}[S_{\alpha_1,\alpha_2}^{\mu_1,\mu_2}(t)]\left\{\mathbb{E}[Z(1)-\text{Var}[Z(1)]]\right\} $, then $Z_1(t)$ shows overdispersion. 
				\item [$(b)$] if $\mathbb{E}[Z(1)]^2 \text{Var}[S_{\alpha_1,\alpha_2}^{\mu_1,\mu_2}(t)]< \mathbb{E}[S_{\alpha_1,\alpha_2}^{\mu_1,\mu_2}(t)]\left\{\mathbb{E}[Z(1)-\text{Var}[Z(1)]]\right\} $, then $Z_1(t)$ shows underdispersion. 
				\item [$(c)$] if $\mathbb{E}[Z(1)]^2 \text{Var}[S_{\alpha_1,\alpha_2}^{\mu_1,\mu_2}(t)]= \mathbb{E}[S_{\alpha_1,\alpha_2}^{\mu_1,\mu_2}(t)]\left\{\mathbb{E}[Z(1)-\text{Var}[Z(1)]]\right\} $, then $Z_1(t)$ shows equidispersion. 
			\end{enumerate}
	\end{enumerate}

	\ifx
	\begin{theorem} 
		The process 
		$$ Z_f(t) =  \sum_{i=1}^{N^{(k)}(D_f(t))} Y_i ~,~ t\geq 0,$$
		where, $\{N^{(k)}(t)\}_{t\geq 0}$ is the PPoK that is independent of $\{D_f(t)\}_{t\geq 0}$ has the following finite dimensional distribution.
		\begin{align} \label{eq:fdd}
		\mathbb{P}[Z_f(t_1)\leq y_1,\ldots, Z_f(t_n)\leq y_n ] =& \prod_{k=1}^{n-1}\int_{-\infty}^{v_k} \sum_{j_1,j_2,\ldots j_n} F^{*{j_n}}(v_n)
		\prod_{l=1}^{n} \mathbb{P}[N^{(k)}(D_f(t_l-t_{l-1}))=j_l]\\
		& \prod_{m=1}^{n-1} f^{*j_m}_{Y_1}(x_m) dx_m, \nonumber
		\end{align}
		where the summation is taken over all non-negative integers $j_i's \geq 0,~ i=1,2,\ldots,n$ and $ v_n = y_n -\sum_{l=1}^{n-1}x_l $.
		
	\end{theorem}
	\fi
	\ifx
	\noindent
	Earlier we gave the martingale characterization for the CPPoK. Now, we extend the well known Watanabe theorem for the CPPoK-I.
	\begin{theorem}[Martingale characterization for time-changed CPPoK]
		Let $\{X(t)\}_{t\geq 0}$ be the point process and $\{D_f(t)\}_{t\geq 0}$ be the L\'evy subordinator with right continuous inverse $\{E_f(t)\}_{t\geq 0}$. Then the process $X(t)$ is subordinated CPPoK \textit{iff} the process $ \{X(t) - \frac{k(k+1)}{2}\lambda \mathbb{E}[Y]D_f(t) \}_{t\geq 0}$ is an $\mathcal{F}_t$, (where $\mathcal{F}_t=\sigma(Z(s), 0 \leq s\leq t)\wedge \sigma(D_f(s)) $) martingale.
		\begin{proof}
			Let $\{X(t)\}_{t\geq 0}$ be the time-changed CPPoK, then $ X(t)\geq 0$. By using \cite[Theorem 2]{fractionalpoissonfields}, we can prove this theorem.
		\end{proof}
	\end{theorem}
\fi
\ifx
\vskip 2ex
	\noindent
	Next, we study the CPPoK-I by taking subordinator as mixture of tempered stable subordinators.
	\subsection{CPPoK time changed by mixtures of tempered stable subordinators}
	The mixtures of tempered stable subordinator (MTSS) $ S_{\alpha_1,\alpha_2}^{\mu_1,\mu_2}(t)$ is a L\'evy process with LT (see \cite{mixedsubordinator})
	$$ \mathbb{E}[e^{-sS_{\alpha_1,\alpha_2}^{\mu_1,\mu_2}(t)}] = e^{-t(c_1((s+\lambda_1)^{\alpha_1} - \lambda_1^{\alpha_1})+c_2((s+\lambda_2)^{\alpha_2} - \lambda_2^{\alpha_2}))} ,~ s>0,$$
	
	\noindent where $c_1+c_2=1, c_1,c_2 \geq 0, \lambda_1,\lambda_2>0 $ are tempering parameter and $ \alpha_1, \alpha_2 \in (0,1)$ are stability index. The function $f(s) = c_1((s+\lambda_1)^{\alpha_1} - \lambda_1^{\alpha_1})+c_2((s+\lambda_2)^{\alpha_2} - \lambda_2^{\alpha_2}) $ is the Lapalce exponent of MTSS.\\
	The mean and variance of MTSS are given as
	\begin{equation}\label{eq:mean}
	 \mathbb{E}[S_{\alpha_1,\alpha_2}^{\mu_1,\mu_2}(t)] = t(c_1\alpha_1\lambda_1^{\alpha_1-1}+c_2\alpha_2\lambda_2^{\alpha_2-1}),
	\end{equation}
	\begin{equation}\label{eq:variance}
	\text{Var}[S_{\alpha_1,\alpha_2}^{\mu_1,\mu_2}(t)]= t(c_1\alpha_1(1-\alpha_1)\lambda_1^{\alpha_1-2}+c_2\alpha_2(1-\alpha_2)\lambda_2^{\alpha_2-2}).
	\end{equation}
	The time-changed CPPoK with MTSS is defined as
	$$ \{Z^1_{f_1}(t)\} = \{Z(S_{\alpha_1,\alpha_2}^{\mu_1,\mu_2}(t))\} = \sum_{i=1}^{N^{(k)}(S_{\alpha_1,\alpha_2}^{\mu_1,\mu_2}(t))} Y_i ,~ t \geq 0 $$
	\begin{corollary}
		When $\alpha_1 = \alpha_2=\alpha,~ \lambda_1=\lambda_2 = 0$ and $c_1+c_2 =1$ then MTSS reduces to $\alpha$ stable subordinator $S_{\alpha}(t)$ and so  $Z^1_{f_1}(t)$ reduces to space fractional CPPoK which can be written as
		$$ \{Z(S_{\alpha}(t))\} = \sum_{i=1}^{N^{(k)}(S_{\alpha}(t))} Y_i ,~ t \geq 0. $$
		When $Y_i's \equiv 1$, then $Z(S_{\alpha}(t))$ reduces to space fractional version of PPoK which we call as SFPPoK and can also be represented as 
		$$N^{(k)}(S_{\alpha}(t)) = N_1(S_{\alpha}(t))+2N_2(S_{\alpha}(t))+\ldots +kN_k(S_{\alpha}(t)),~ t\geq 0, $$
		where $N_i(t), i=1,2,\ldots,k$ are independent Poisson process with same intensity  $\lambda>0$. \\
		Its probability generating function is given as 
		$$ G_{\alpha}(u,t) = e^{[-t\lambda^{\alpha}(k-\sum_{j=1}^{k}u^j)^{\alpha}]},~ t\geq 0.$$
	\end{corollary}
	\begin{corollary}
	     When $\alpha_1 = \alpha_2=\alpha,~ \lambda_1= \lambda,\lambda_2 = 0$ and $c_1= 1, c_2 =0$ then MTSS reduces to  tempered $\alpha$ stable subordinator $S_{\alpha}^{\lambda}(t)$ and so  $Z^1_{f_1}(t)$ reduces to tempered space fractional CPPoK with tempering parameter $\lambda$
	     $$ \{Z(S_{\alpha}^{\lambda}(t))\} = \sum_{i=1}^{N^{(k)}(S_{\alpha}^{\lambda}(t))} Y_i ,~ t \geq 0. $$
	     Substituting $Y_i's \equiv 1$, then it becomes tempered space fractional version of PPoK, represented as TSFPPoK.
	\end{corollary}
	\noindent
	The mean and variance of the CPPoK-I can be obtained by putting  \eqref{eq:mean},\eqref{eq:variance} in Theorem \ref{thm:dis}
	\ifx
	\begin{proposition} Let $\{Z(S_{\alpha_1,\alpha_2}^{\mu_1,\mu_2}(t))\}_{t\geq 0}$ be the CPPoK time changed with MTSS. Then the mean and covariance function of this process is given as
		\begin{enumerate}[(i)]
			\item $\mathbb{E}[Z(S_{\alpha_1,\alpha_2}^{\mu_1,\mu_2}(t))]= \frac{k(k+1)}{2} \lambda \mathbb{E}[Y] \mathbb{E}[S_{\alpha_1,\alpha_2}^{\mu_1,\mu_2}(t)].$
			\item $ Cov[Z(S_{\alpha_1,\alpha_2}^{\mu_1,\mu_2}(s)),Z(S_{\alpha_1,\alpha_2}^{\mu_1,\mu_2}(t))] = \frac{k(k+1)}{2}\lambda \text{Var}[Y]\mathbb{E}[Z(S_{\alpha_1,\alpha_2}^{\mu_1,\mu_2}(s))] +  \frac{k(k+1)(2k+1)}{6}\lambda \mathbb{E}[Y]^2\\  \mathbb{E}[Z(S_{\alpha_1,\alpha_2}^{\mu_1,\mu_2}(s))] +  ( \frac{k(k+1)}{2} \lambda \mathbb{E}[Y])^2 \text{Var}[Z(S_{\alpha_1,\alpha_2}^{\mu_1,\mu_2}(s))].$
			\end{enumerate}
	\end{proposition}
	\noindent	
	On putting $s=t$ in $(ii)$, we get variance.
	\noindent
	Now, we evaluate $ \text{Var}[Z(S_{\alpha_1,\alpha_2}^{\mu_1,\mu_2}(t))] - \mathbb{E}[Z(S_{\alpha_1,\alpha_2}^{\mu_1,\mu_2}(t))]$,
	\begin{eqnarray*}
		&=&   \frac{k(k+1)}{2} \lambda \mathbb{E}[S_{\alpha_1,\alpha_2}^{\mu_1,\mu_2}(t)]\left[\text{Var}[Y]-\mathbb{E}[Y]+ \frac{2k+1}{3}\mathbb{E}[Y]^2 \right] +
		\\ && \left(\frac{k(k+1)}{2} \lambda \mathbb{E}[Y]\right)^2 \text{Var}[S_{\alpha_1,\alpha_2}^{\mu_1,\mu_2}(t)]\\ &>& 0.
	\end{eqnarray*}
	Hence it exhibits overdispersion.\\
	\fi
	\fi
	
	\subsection*{Long-range dependence}
	Now we analyze the LRD property for $ TCPPoK(\lambda,H,S_{\alpha_1,\alpha_2}^{\mu_1,\mu_2})$.

	\begin{theorem}
		The  $\{Z_1(t)\}_{t\geq 0}$ has the LRD property.
	\end{theorem} 
	\begin{proof}
		We have $ \mathbb{E}[S_{\alpha_1,\alpha_2}^{\mu_1,\mu_2}(t)^n] \sim (c_1\alpha_1\lambda_1^{\alpha_1-1}+c_2\alpha_2\lambda_2^{\alpha_2-1})^n t^n $, as $t\rightarrow \infty $, from \cite{mixedsubordinator}.
	Therefore
		\begin{equation}\label{eq:var}
		\text{Var}[Z_1(t)] 
		\sim  \mathbb{E}[S_{\alpha_1,\alpha_2}^{\mu_1,\mu_2}(s)]\text{Var}[Z(1)] \sim Kt,
		\end{equation}
		where $K= (c_1\alpha_1\lambda_1^{\alpha_1-1}+c_2\alpha_2\lambda_2^{\alpha_2-1})\text{Var}[Z(1)].$\\
		Let $ 0<s<t<\infty $, then
		\begin{eqnarray*}
			Corr[Z_1(s),Z_1(t)] &=& \frac{Cov[Z_1(s),Z_1(t)]}{\sqrt{\text{Var}[Z_1(s)]\text{Var}[Z_1(t)]}},\\
			&\sim& \frac{Cov[Z_1(s),Z_1(t)]}{t^{1/2}\sqrt{K}}, \text{from \eqref{eq:var}} \\
			&= & c(s)t^{-1/2},
		\end{eqnarray*}
		where $ c(s) = \frac{Cov[Z_1(s),Z_1(t)]}{\sqrt{K}} >0. $ Hence from the Definition \ref{Def:LRD}, $TCPPoK(\lambda,H,S_{\alpha_1,\alpha_2}^{\mu_1,\mu_2})$ captures LRD property.
	\end{proof}

	\ifx
	\begin{theorem}
		(Law of iterated logarithm) Let the Laplace exponent $f(s)$ of the subordinator $ D_f(t)$ be regularly varying at $ 0+$ with index $ \alpha \in (0,1)$ . Then ({\color{red} need to be checked })
		$$ \liminf_{t \rightarrow \infty} \frac{Q_f^{(1)}(t)}{g(t)} =  \frac{k(k+1)}{2}\lambda \mathbb{E}[Y] \alpha (1-\alpha)^{\frac{1-\alpha}{\alpha}}~~ a.s., $$
		where,
		$$ g(t) = \frac{loglogt }{f^{-1}(t^{-1}loglogt)}(t>e)$$
	\end{theorem}
	\textbf{Proof:-} We have $ Q_f^{(1)}(t) = Z(D_f(t))$\\
	Now,
	\begin{eqnarray*}
		\liminf_{t \rightarrow \infty} \frac{Q_f^{(1)}(t)}{g(t)} &= & \liminf_{t \rightarrow \infty} \frac{Z(D_f(t))}{g(t)} \\ &=&
		\liminf_{t \rightarrow \infty} \frac{Z(D_f(t))}{D_f(t)}\frac{D_f(t)}{g(t)} \\&=& \frac{k(k+1)}{2}\lambda \mathbb{E}[Y]  \liminf_{t \rightarrow \infty} \frac{D_f(t)}{g(t)} ,~~
		law~ of~ large~ number \\ &=& \frac{k(k+1)}{2}\lambda \mathbb{E}[Y] \alpha (1-\alpha)^{\frac{1-\alpha}{\alpha}} 
	\end{eqnarray*}
	\fi
	
	\subsection{CPPoK time changed by the first exit time of mixtures of tempered stable subordinators}
	In this subsection, we consider $CPPoK(\lambda,H)$ subordinated with first exit time of MTSS and discuss the asymptotic behavior of its moments. \\
	The first exit time of L\'evy subordinator $S_{\alpha_1,\alpha_2}^{\mu_1,\mu_2}(t)$ also known as inverse subordinator is defined as
	$$ E_{\alpha_1,\alpha_2}^{\mu_1,\mu_2}(t)=\inf\{r\geq 0:S_{\alpha_1,\alpha_2}^{\mu_1,\mu_2}(r)>t\},~ t\geq 0.$$
	\begin{definition}
	   Let $\{Z(t)\}_{t\geq 0}$ be the $CPPoK(\lambda,H)$ as discussed in Definition \ref{def:cppok}, then the subordinated CPPoK with first exit time of MTSS, denoted as $TCPPoK(\lambda,H,E_{\alpha_1,\alpha_2}^{\mu_1,\mu_2})$ is defined as
		$$ Z_2(t) = Z(E_{\alpha_1,\alpha_2}^{\mu_1,\mu_2}(t))= \sum_{i=1}^{N^{(k)}(E_{\alpha_1,\alpha_2}^{\mu_1,\mu_2}(t))} Y_i ,~ t \geq 0 $$
		where the process $ \{Z(t)\}_{t\geq 0} $ is independent from $\{ E_{\alpha_1,\alpha_2}^{\mu_1,\mu_2}(t)\}_{t\geq 0}.$ 
	\end{definition}
\noindent
	The \textit{FDD} of $TCPPoK(\lambda,H,E_{\alpha_1,\alpha_2}^{\mu_1,\mu_2})$ is given as
	\begin{equation*}  
    F_{Z_2(t_1),\ldots,Z_2(t_n)}(y_1,\ldots,y_n)= \prod_{i=1}^{n}\int_{x_i=0}^{\infty}
		F'(y_1,\ldots,y_n) g'(x_1,\ldots,x_n)dx_1\ldots dx_n,
    \end{equation*}
	where $F'(y_1,\ldots,y_n) = F_{Z(x_1),\ldots,Z(x_n)}(y_1,\ldots,y_n)$ is the \textit{FDD} of $CPPoK(\lambda,H)$, and $g'(x_1,\ldots,x_n)$ is the joint density function of $E_{\alpha_1,\alpha_2}^{\mu_1,\mu_2}(t_1),\ldots,E_{\alpha_1,\alpha_2}^{\mu_1,\mu_2}(t_n).$
	\begin{theorem}
	 The mean and covariance function of $TCPPoK(\lambda,H,E_{\alpha_1,\alpha_2}^{\mu_1,\mu_2}) $ are given as
	\begin{enumerate}[(i)]
		\item $\mathbb{E}[Z_2(t)] = \frac{k(k+1)}{2} \lambda \mathbb{E}[Y] \mathbb{E}[E_{\alpha_1,\alpha_2}^{\mu_1,\mu_2}(t)]$,
		\item $ \mbox{Cov}[Z_2(s),Z_2(t)] = \text{Var}[Z(1)] \mathbb{E}[E_{\alpha_1,\alpha_2}^{\mu_1,\mu_2}(s)] +\mathbb{E}[Z(1)]^2 \mbox{Cov}[E_{\alpha_1,\alpha_2}^{\mu_1,\mu_2}(s),E_{\alpha_1,\alpha_2}^{\mu_1,\mu_2}(t)].$
	\end{enumerate}
\end{theorem}
\begin{proof}
We can prove this in a similar manner as given in Theorem \ref{thm:dis}.	
\end{proof}
\ifx
	\begin{eqnarray*}
		\mbox{Cov}[Q_f^{(2)}(s),Q_f^{(2)}(t)] &=& \left[ \frac{k(k+1)}{2}\lambda  \text{Var}(Y) + \frac{k(k+1)(2k+1)}{6}\lambda  \mathbb{E}[Y]^2 \right]\mathbb{E}[E_{\alpha_1,\alpha_2}^{\mu_1,\mu_2}(s)\}] \\ &&
		+ \left(\frac{k(k+1)}{2} \lambda \mathbb{E}[Y] \right)^2 \mbox{Cov}[E_{\alpha_1,\alpha_2}^{\mu_1,\mu_2}(s),E_{\alpha_1,\alpha_2}^{\mu_1,\mu_2}(t)].
	\end{eqnarray*}
\fi
\noindent
	Now, we discuss the asymptotic behavior of moments of $TCPPoK(\lambda,H,E_{\alpha_1,\alpha_2}^{\mu_1,\mu_2})$. First we need the following definition (see \cite{bertoin,Taqqu2010}).
	\begin{theorem}(Tauberian Theorem) \label{thm:taub}
		Let $ l:(0,\infty)\rightarrow (0,\infty)$ be a slowly varying function at $0 $ (respectively $\infty $) and let $\rho \geq 0$. Then for a function $ U:(0,\infty)\rightarrow (0,\infty) $, the following are equivalent.
		\begin{enumerate}[(i)]
			\item $ U(x) \sim x^{\rho}l(x)/ \Gamma(1+\rho),~~~x\rightarrow 0 ~(respectively ~x\rightarrow \infty).$
			\item $\tilde{U}(s) \sim s^{-\rho-1}l(1/s),~~~s \rightarrow \infty~ (respectively ~s\rightarrow 0), $ where $\tilde{U}(s) $ is the LT of $U(x).$
		\end{enumerate}
	\end{theorem}
	\noindent We know that LT of $p$th order moment of inverse subordinator $ \{E_f(t)\}_{t\geq 0} $  is given by  (see \cite{TCFPP-pub})
	$$ \mathscr{L}[\mathbb{E}(E_f(t))^p] = \frac{\Gamma (1+p)}{s(f(s))^p},~ p>0$$
	where $ f(s)$ is the corresponding Bernstein function associated with L\'evy subordinator.
	\begin{proposition}
		The asymptotic form of mean and variance of $ \{Z_2(t)\}_{t\geq 0}$ is given as 
		$$\mathbb{E}[Z_2(t)] \sim \frac{k(k+1)}{2} \lambda \mathbb{E}[Y] \frac{t}{c_1\alpha_1 \lambda_1^{\alpha_1 -1}+c_2\alpha_2 \lambda_2^{\alpha_2 -1}}, ~~as~ t\rightarrow \infty. $$
		$$ \text{Var}[Z_2(t)] \sim \frac{k(k+1)}{2}\lambda \left[\text{Var}(Y) + \frac{(2k+1)}{3}  \mathbb{E}[Y]^2 \right] \frac{t}{c_1\alpha_1 \lambda_1^{\alpha_1 -1}+c_2\alpha_2 \lambda_2^{\alpha_2 -1}}, ~~as~ t\rightarrow \infty. $$
	\end{proposition}
	\begin{proof} 
		Let $\tilde{M}(s)$ be the LT of $\mathbb{E}[E_{\alpha_1,\alpha_2}^{\mu_1,\mu_2}(t)]$. Then
		\begin{eqnarray*}
	\tilde{M}(s) = \mathscr{L}[\mathbb{E}[E_{\alpha_1,\alpha_2}^{\mu_1,\mu_2}(t)]] &=&\frac{1}{s(c_1((s+\lambda_1)^{\alpha_1}-\lambda_1^{\alpha_1})+c_2((s+\lambda_2)^{\alpha_2}-\lambda_2^{\alpha_2}))}\\
	&\sim & \frac{1}{s^2(c_1\alpha_1 \lambda_1^{\alpha_1 -1} + (c_2\alpha_2 \lambda_2^{\alpha_2 -1})}, ~\text{as $s\rightarrow 0$, from \cite{mixedsubordinator}}.
		\end{eqnarray*}
	Then, by using Theorem \ref{thm:taub}, we have that
\begin{eqnarray*}
\mathbb{E}[Z_2(t)]  &=& \frac{k(k+1)}{2} \lambda \mathbb{E}[Y] \mathbb{E}[E_{\alpha_1,\alpha_2}^{\mu_1,\mu_2}(t)]\\
&\sim & \frac{\lambda k(k+1)}{2} \mathbb{E}[Y]\frac{t}{c_1\alpha_1 \lambda_1^{\alpha_1 -1}+c_2\alpha_2 \lambda_2^{\alpha_2 -1}}, ~~as~ t\rightarrow \infty. 
\end{eqnarray*}	
In the similar manner, we can get the expression for variance.
\end{proof}	
	
\ifx
	  \begin{align*}
		\mathbb{E}[Z_{f_1}^{(2)}(t)] =& \frac{k(k+1)}{2} \lambda \mathbb{E}[Y] \mathbb{E}[E_{S_{\alpha_1,\alpha_2}^{\mu_1,\mu_2}}(t)\}]\\ &
		\text{by using equation \eqref{asymp} in theorem \eqref{thm:taub}, we get}\\
		\sim & \frac{k(k+1)}{2} \lambda \mathbb{E}[Y] \frac{t}{c_1\alpha_1 \lambda_1^{\alpha_1 -1}+c_2\alpha_2 \lambda_2^{\alpha_2 -1}}, ~~as~ t\rightarrow \infty.
		\end{align*}
		and
		\begin{align*}
		\text{Var}[Z_{f_1}^{(2)}(t)] =& \left[ \frac{k(k+1)}{2}\lambda  \text{Var}(Y) + \frac{k(k+1)(2k+1)}{6}\lambda  \mathbb{E}[Y]^2 \right]\mathbb{E}[E_{S_{\alpha_1,\alpha_2}^{\mu_1,\mu_2}}(t)\}] \\ &
		+ \left(\frac{k(k+1)}{2} \lambda \mathbb{E}[Y] \right)^2 \text{Var}[E_{S_{\alpha_1,\alpha_2}^{\mu_1,\mu_2}}(t)] \\
		\sim & \frac{k(k+1)}{2}\lambda \left[\text{Var}(Y) + \frac{(2k+1)}{3}  \mathbb{E}[Y]^2 \right] \frac{t}{c_1\alpha_1 \lambda_1^{\alpha_1 -1}+c_2\alpha_2 \lambda_2^{\alpha_2 -1}}, ~~as~ t\rightarrow \infty.
		\end{align*}
		\fi
		\noindent		
{\bf Acknowledgement}. The authors would like to thank Aditya Maheshwari(IIM, Indore) for being part of various discussions on this topic.
 
	\bibliographystyle{abbrv}
	\bibliography{researchbib}
\end{document}